\documentclass[11pt]{article}
\usepackage{amsfonts,amsmath,amssymb}
\usepackage{mathrsfs,mathtools,url}
\usepackage[T1]{fontenc}
\usepackage{graphicx,tikz,color}
\usepackage{cite,url}

\newtheorem{theorem}{Theorem}[section]

\newtheorem{lemma}[theorem]{Lemma}

\newcommand{\proof}{\noindent{\bf Proof.\ }}
\newcommand{\qed}{\hfill $\square$ \bigskip}

\newcommand{\dstart}{\gamma_{\rm g}}
\newcommand{\sstart}{\gamma_{\rm g}^\prime}
\newcommand{\tdstart}{\gamma_{{\rm{tg}}}}
\newcommand{\tsstart}{\gamma_{{\rm{tg}}}^\prime}
\newcommand{\gammat}{\gamma_{{\rm t}}}
\newcommand{\grundy}{\gamma_{{\rm{gr}}}}
\newcommand{\grundyt}{\gamma_{{\rm{gr}}}^{\rm t}}

\begin{document}

\title{Domination game and total domination game played on Sierpi\'{n}ski graphs}
\date{}

\author{
Tanja Dravec $^{a,b,}$\thanks{Email: \texttt{tanja.dravec@um.si}}
\and Daniel P. Johnston $^{c,}$\thanks{Email: \texttt{daniel.johnston@trincoll.edu}}
\and Sandi Klav\v{z}ar $^{a,b,d,}$\thanks{Email: \texttt{sandi.klavzar@fmf.uni-lj.si}}
}
\maketitle

\begin{center}
$^a$ Faculty of Natural Sciences and Mathematics, University of Maribor, Slovenia \\
\medskip
$^b$ Institute of Mathematics, Physics and Mechanics, Ljubljana, Slovenia \\
\medskip
$^c$ Department of Mathematics, Trinity College, Hartford, CT, USA\\
\medskip
$^d$ Faculty of Mathematics and Physics, University of Ljubljana, Slovenia
\end{center}
\maketitle

\begin{abstract}
The game domination numbers $\gamma_{\rm g}$ and $\gamma_{\rm g}^\prime$, and the game total domination numbers $\gamma_{{\rm{tg}}}$ and $\gamma_{{\rm{tg}}}^\prime$ are investigated on Sierpi\'{n}ski graphs $S_p^n$. For the domination game, the trivial bounds arising from the known domination number and the Grundy domination number of $S_p^n$ are significantly improved by proving that $(2p-3)p^{n-2}\leq \gamma_{\rm g}(S_p^n), \gamma_{\rm g}^\prime(S_p^n)
\leq (2p-2)p^{n-2}$. For the total domination game the bounds $\gamma_{{\rm{tg}}}(S_p^n), \gamma_{{\rm{tg}}}^\prime(S_p^n) \geq (2p-2)p^{n-2}$ are established. 
\end{abstract}

\noindent
{\bf Keywords:} domination game; total domination game; Sierpi\'{n}ski graph; Grundy domination number; total Grundy domination number

\medskip\noindent
{\bf AMS Subj.\ Class.\ (2020)}: 05C57, 05C69

\section{Introduction}

Since its introduction in 2010~\cite{bresar-2010}, the domination game has met with a very wide response and has been studied in great detail so far. The state of the field up to 2021 was summarized in the book~\cite{book-2021}. The domination game research is still very topical, in particular, Versteegen~\cite{versteegen-2024} proved the $3/5$-Conjecture for the domination game posed in~\cite{kinnersley-2013}, while Portier and Versteegen~\cite{portier-2025} proved the $3/4$-Conjecture for the total domination game posed in~\cite{henning-2017}. For selected additional recent achievements see~\cite{brito-2026, bujtas-2022, forcan-2022, fuhrer-2026+, irsic-2025, james-2023, wor-2024}. 

Sierpi\'{n}ski graphs~\cite{km-1997} were introduced in connection with various problems and applications, including the Tower of Hanoi game, physics, interconnection networks, and topology; see the extensive survey~\cite{kmz-2017} on the developments up to 2017, and the following selected subsequent references~\cite{att-2026, hs-2022, jos-2025, mcs-2023, liu-2021, yang-2025}. A characteristic feature of Sierpi\'{n}ski graphs is their fractal-like structure, which makes them a natural discrete analogue of fractals. 

From the algorithmic perspective, domination is a fundamental graph-theoretic problem that is computationally hard in general~\cite{GJ-79}, but the domination number can be determined efficiently for several important graph classes. In particular, for Sierpi\'{n}ski graphs $S_p^n$, exact formulas for the domination number have been established in~\cite[Theorem 3.8]{kmp-2002}: if $p\ge 1$ and $n \ge 0$, then
\begin{equation}
\label{eq:gamma}
\gamma (S_p^n) = \frac{p^n + p^{[n\ {\rm even}]}}{p+1}\,,
\end{equation}
where $[n\ {\rm even}]$ denotes the {\em Iverson bracket} which is, for a given statement $\mathfrak{S}$, defined as
$[\mathfrak{S}] = 1$ if $\mathfrak{S}$ is true, and 
$[\mathfrak{S}] = 0$ otherwise. Moreover, the Grundy domination number $\gamma_{\mathrm{gr}}$ of Sierpi\'{n}ski graphs has been determined in~\cite[Theorem 5]{bgk-2016} as follows: if $n, p\ge 1$, then 
\begin{equation}
\label{eq:grundy}
\gamma_{\mathrm{gr}}(S_p^n)=p^{n-1}+\frac{p(p^{n-1}-1)}{2}\,.
\end{equation}
Motivated by~\eqref{eq:gamma} and~\eqref{eq:grundy}, we investigate in this paper the domination game played on $S_p^n$. Namely, formula~\eqref{eq:gamma} can be considered as the game domination number for the case when Dominator is the solitary player, while formula~\eqref{eq:grundy} is the solution for the case when Staller is the only player. Hence it is natural to ask what happens when Dominator and Staller play simultaneously. Having in mind the extensive investigation of the domination game as well as of Sierpi\'{n}ski graphs it is rather surprising that this problem has not yet been studied.  

In the next section we formally introduce the concepts investigated and recall earlier results which we need. In Section~\ref{sec:dom-game} we study the domination game on Sierpi\'{n}ski graphs. In particular, we improve the trivial lower bound given by the domination number and the trivial upper bound given by twice the domination number minus one, or equivalently, by the Grundy domination number. Our main result in this section establishes that, for $p\geq 3$ and $n \geq 2$, both $\dstart(S_p^n)$ and $\sstart(S_p^n)$ are bounded from below by 
$(2p-3)p^{n-2}$ and from above by $(2p-2)p^{n-2}.$ In Section~\ref{sec:total-dom-game} we then consider the total domination game and prove that if $p\geq 3$, then 
$\tdstart(S_p^n) \geq (2p-2)p^{n-2}$ and
$\tsstart(S_p^n) \geq (2p-2)p^{n-2}$.

\section{Preliminaries}

For a positive integer $p$ we will use the convention $[p] = \{1,\dots, p\}$. Let $G = (V(G), E(G))$ be a graph. A set $D \subseteq V(G)$ is {\em dominating} if each vertex from $V(G)\setminus D$ has a neighbour in $D$ and is {\em total dominating} if each vertex from $V(G)$ has a neighbour in $D$. The {\em domination number} $\gamma(G)$ of $G$  is the minimum cardinality of a dominating set of $G$ and the {\em total domination number} $\gammat(G)$ of $G$  is the minimum cardinality of a total dominating set of $G$. A {\em legal dominating sequence} in $G$ is a sequence of vertices such that each vertex from the sequence dominates at least one vertex of $G$ which is not dominated by the preceding elements in the sequence. The maximum length of a legal dominating sequence in $G$ is the {\em Grundy domination number} of $G$ denoted by $\grundy(G)$. This invariant was studied for the first time in~\cite{bgmrr-2014}. The {\em Grundy total domination number} $\grundyt(G)$ is defined analogously for total domination sequences and was introduced in~\cite{bresar-2016}.

The {\em domination game} is played on a graph $G$ by Dominator and Staller, which alternately choose vertices of $G$ such that each newly selected vertex dominates at least one vertex not previously dominated. The game is over when there is no vertex available to play. At that time the set of vertices selected, say $X$, forms a dominating set of $G$. Dominator's aim is to finish the game such that $|X|$ is as small as possible, Staller's goal is just the opposite. A {\em D-game} is the game in which Dominator has the first move and an {\em S-game} is the game in which Staller has the first move. The {\em game domination number}, $\dstart(G)$, of $G$ is the number of moves in a D-game when both players play optimally. The {\em Staller start game domination number}, $\sstart(G)$, of $G$ is defined analogously for the S-game. The {\em total domination game} is played analogously as the domination game, except that each newly selected vertex totally dominates at least one vertex not previously totally dominated. The corresponding invariants are denoted by $\tdstart(G)$ and $\tsstart(G)$.

For a set $A \subseteq V(G)$, a \emph{partially dominated graph}, $G|A$ is a graph in which vertices from $A$ are considered already dominated. If $G|A$ is a partially dominated graph, then $\dstart(G|A)$ and $\sstart(G|A)$ denote the optimal number of moves remaining in D-game and S-game, respectively. Analogously $\tdstart(G|A)$ and $\tsstart(G|A)$ are defined.

Next, we recall several known results on the game domination number and the Staller-start game domination number that will be used throughout the paper.

\begin{theorem} {\rm \cite{bresar-2010}}
\label{thm:gamma}
    If $G$ is a graph, then $\dstart(G) \leq 2\gamma(G)-1$.
\end{theorem}

\begin{lemma}{\rm \cite[Lemma~2.1]{kinnersley-2013}}
  \label{lem:continuation}
  {\rm (Continuation Principle)}
  Let $G$ be a graph, and let $A,B\subseteq V(G)$. If $B\subseteq A$, then
  $\dstart(G|A)\le \dstart(G|B)$ and $\sstart(G|A)\le \sstart(G|B)$.
\end{lemma}

Using the Continuation Principle, the following important result can be deduced. 

\begin{theorem} {\rm \cite{bresar-2010, kinnersley-2013}} 
\label{thm:dif1}
    If $G$ is a graph, then $|\dstart(G)-\sstart(G)| \leq 1$.
\end{theorem}

Let $G$ be a graph and $uv$ be an edge in $G$. Let $G_{uv}|\{u',v'\}$ be  the graph obtained from $G - uv$ by adding two new vertices $u'$ and $v'$, adding the two edges $uu'$ and $vv'$, and declaring that both $u'$ and $v'$ are dominated. For notational convenience, we write $G_{uv}$ rather than $G_{uv}|\{u',v'\}$.

\begin{theorem}
{\rm (Cutting Lemma)\cite[Theorem~3.1]{dorbec_cutting_2019}}
\label{thm:cutting-lemma}
Let $G$ be a graph, and let $A,B \subseteq V(G)$ where $B\subseteq A$. If $uv$ is an edge of $G$, then
$\dstart(G|A) \le \dstart(G_{uv}|B)$ and $\sstart(G|A) \le \sstart(G_{uv}|B)$.
\end{theorem}

If $p \geq 3$ and $n \geq 1$, then the {\em Sierpi\'nski graph} $S_p^n$ has the vertex set $V (S_p^n) = [p]^n$. We will simplify the notation of a vertex $(i_1, \ldots, i_n)\in V(S_p^n)$ to $i_1\cdots i_n$. Vertices $i_1\cdots i_n$ and $j_1 \cdots j_n$ of $S_p^n$ are adjacent if there exists an index $h  \in [n]$, such that (i) $i_t = j_t$ for every $t < h$, (ii) $i_h \neq j_h$, and (iii) $i_t = j_h$ and $j_t = i_h$ for every $t > h$. Define the set of vertices  
$${\rm Ex}(S_p^n) = \{i^n:\ i\in [p]\},$$ 
the elements of which are referred to as the {\em extreme vertices} of $S_p^n$. If $s \in [p]^{n-k}$, where $k \in [n-1]$, then the subgraph of $S_p^n$ induced by the vertices of the form $\{st:\ t \in [p]^k\}$, is isomorphic to $S_p^k$ and denoted by $\underline{s}S_p^{k}$.  For $i\in [p]$, the notation $\underline{i}S_p^{n-1}$ is simplified to $iS_p^{n-1}$ in which case we have $iS_p^{n-1} \cong S_p^{n-1}$.

For $p\ge 3$, let $\Gamma_p$ be the graph obtained from $S_p^2$ by respectively adding a pendant vertex to each of its $p$ extreme vertices. The vertices of $\Gamma_p$ corresponding to the vertices of $S_p^2$ will be labeled using the convention for $S_p^2$, and the vertices attached to the extreme vertices will be denoted by $x_1, \dots, x_p$, where $x_i$ is adjacent to $ii$. Set also $X_p = \{x_1,\dots, x_p\}$. In addition, let $\Gamma_p^-$ be the graph obtained from $\Gamma_p$ by removing the vertex $x_p$ and let $X_p^- = \{x_1,\dots, x_{p-1}\}$. In Figure~\ref{fig:Gamma_4}, the graphs $\Gamma_4$ and $\Gamma_4^-$ are presented.

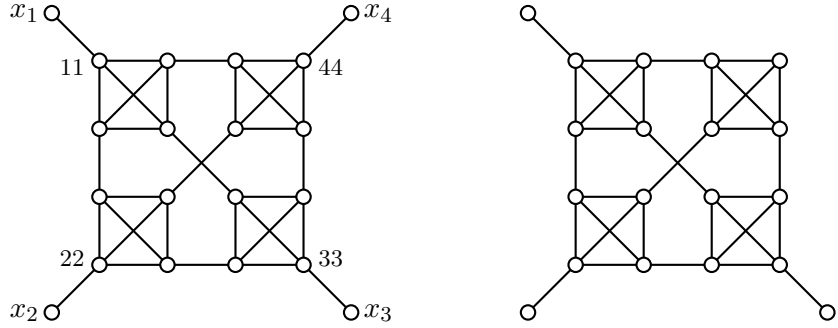
\begin{figure}[ht!]
\begin{center}
\begin{tikzpicture}[scale=0.9,style=thick,x=1cm,y=1cm]
\def\vr{3pt}
%
%
\begin{scope}[xshift=0cm, yshift=0cm] 
%
\begin{scope}[xshift=0cm, yshift=0cm] 
\coordinate(x1) at (0,0);
\coordinate(x2) at (1,0);
\coordinate(x3) at (1,1);
\coordinate(x4) at (0,1);
\coordinate(xx1) at (-0.7,-0.7);
\coordinate(xx2) at (3.7,-0.7);
\coordinate(xx3) at (3.7,3.7);
\coordinate(xx4) at (-0.7,3.7);
\draw (x1) -- (x2) -- (x3) -- (x4) -- (x1) -- (x3);
\draw (x2) -- (x4);
\draw (0,1) -- (0,2);
\draw (1,3) -- (2,3);
\draw (3,2) -- (3,1);
\draw (1,0) -- (2,0);
\draw (1,1) -- (2,2);
\draw (2,1) -- (1,2);
\draw (x1) -- (xx1);
\draw (3,0) -- (xx2);
\draw (3,3) -- (xx3);
\draw (0,3) -- (xx4);
\draw(x1)[fill=white] circle(\vr);
\draw(x2)[fill=white] circle(\vr);
\draw(x3)[fill=white] circle(\vr);
\draw(x4)[fill=white] circle(\vr);
\draw(xx1)[fill=white] circle(\vr);
\draw(xx2)[fill=white] circle(\vr);
\draw(xx3)[fill=white] circle(\vr);
\draw(xx4)[fill=white] circle(\vr);
\node at (-0.4,2.9) {{\footnotesize $11$}};
\node at (-0.4,0.1) {{\footnotesize $22$}};
\node at (3.4,0.1) {{\footnotesize $33$}};
\node at (3.4,2.9) {{\footnotesize $44$}};
\node at (-1.1,3.7) {$x_1$};
\node at (-1.1,-0.7) {$x_2$};
\node at (4.1,-0.7) {$x_3$};
\node at (4.1,3.7) {$x_4$};

\end{scope}
%
%
\begin{scope}[xshift=2cm, yshift=0cm] 
\coordinate(x1) at (0,0);
\coordinate(x2) at (1,0);
\coordinate(x3) at (1,1);
\coordinate(x4) at (0,1);
\draw (x1) -- (x2) -- (x3) -- (x4) -- (x1) -- (x3);
\draw (x2) -- (x4);
\draw(x1)[fill=white] circle(\vr);
\draw(x2)[fill=white] circle(\vr);
\draw(x3)[fill=white] circle(\vr);
\draw(x4)[fill=white] circle(\vr);
\end{scope}
%
%
\begin{scope}[xshift=0cm, yshift=2cm] 
\coordinate(x1) at (0,0);
\coordinate(x2) at (1,0);
\coordinate(x3) at (1,1);
\coordinate(x4) at (0,1);
\draw (x1) -- (x2) -- (x3) -- (x4) -- (x1) -- (x3);
\draw (x2) -- (x4);
\draw(x1)[fill=white] circle(\vr);
\draw(x2)[fill=white] circle(\vr);
\draw(x3)[fill=white] circle(\vr);
\draw(x4)[fill=white] circle(\vr);
\end{scope}
%
%
\begin{scope}[xshift=2cm, yshift=2cm] 
\coordinate(x1) at (0,0);
\coordinate(x2) at (1,0);
\coordinate(x3) at (1,1);
\coordinate(x4) at (0,1);
\draw (x1) -- (x2) -- (x3) -- (x4) -- (x1) -- (x3);
\draw (x2) -- (x4);
\draw(x1)[fill=white] circle(\vr);
\draw(x2)[fill=white] circle(\vr);
\draw(x3)[fill=white] circle(\vr);
\draw(x4)[fill=white] circle(\vr);
\end{scope}
\end{scope}

%
%
\begin{scope}[xshift=7cm, yshift=0cm] 
%
\begin{scope}[xshift=0cm, yshift=0cm] 
\coordinate(x1) at (0,0);
\coordinate(x2) at (1,0);
\coordinate(x3) at (1,1);
\coordinate(x4) at (0,1);
\coordinate(xx1) at (-0.7,-0.7);
\coordinate(xx2) at (3.7,-0.7);
\coordinate(xx4) at (-0.7,3.7);
\draw (x1) -- (x2) -- (x3) -- (x4) -- (x1) -- (x3);
\draw (x2) -- (x4);
\draw (0,1) -- (0,2);
\draw (1,3) -- (2,3);
\draw (3,2) -- (3,1);
\draw (1,0) -- (2,0);
\draw (1,1) -- (2,2);
\draw (2,1) -- (1,2);
\draw (x1) -- (xx1);
\draw (3,0) -- (xx2);
\draw (0,3) -- (xx4);
\draw(x1)[fill=white] circle(\vr);
\draw(x2)[fill=white] circle(\vr);
\draw(x3)[fill=white] circle(\vr);
\draw(x4)[fill=white] circle(\vr);
\draw(xx1)[fill=white] circle(\vr);
\draw(xx2)[fill=white] circle(\vr);
\draw(xx4)[fill=white] circle(\vr);

\end{scope}
%
%
\begin{scope}[xshift=2cm, yshift=0cm] 
\coordinate(x1) at (0,0);
\coordinate(x2) at (1,0);
\coordinate(x3) at (1,1);
\coordinate(x4) at (0,1);
\draw (x1) -- (x2) -- (x3) -- (x4) -- (x1) -- (x3);
\draw (x2) -- (x4);
\draw(x1)[fill=white] circle(\vr);
\draw(x2)[fill=white] circle(\vr);
\draw(x3)[fill=white] circle(\vr);
\draw(x4)[fill=white] circle(\vr);
\end{scope}
%
%
\begin{scope}[xshift=0cm, yshift=2cm] 
\coordinate(x1) at (0,0);
\coordinate(x2) at (1,0);
\coordinate(x3) at (1,1);
\coordinate(x4) at (0,1);
\draw (x1) -- (x2) -- (x3) -- (x4) -- (x1) -- (x3);
\draw (x2) -- (x4);
\draw(x1)[fill=white] circle(\vr);
\draw(x2)[fill=white] circle(\vr);
\draw(x3)[fill=white] circle(\vr);
\draw(x4)[fill=white] circle(\vr);
\end{scope}
%
%
\begin{scope}[xshift=2cm, yshift=2cm] 
\coordinate(x1) at (0,0);
\coordinate(x2) at (1,0);
\coordinate(x3) at (1,1);
\coordinate(x4) at (0,1);
\draw (x1) -- (x2) -- (x3) -- (x4) -- (x1) -- (x3);
\draw (x2) -- (x4);
\draw(x1)[fill=white] circle(\vr);
\draw(x2)[fill=white] circle(\vr);
\draw(x3)[fill=white] circle(\vr);
\draw(x4)[fill=white] circle(\vr);
\end{scope}
\end{scope}

\end{tikzpicture}
\caption{Graphs $\Gamma_4$ (left) and $\Gamma_4^-$ (right).}
\label{fig:Gamma_4}
\end{center}
\end{figure}

\section{Domination game}
\label{sec:dom-game}

In this section we study the (Staller start) game domination number of Sierpi\'{n}ski graphs. First, combining~\eqref{eq:gamma} with Theorem~\ref{thm:gamma} we get 
\begin{equation}
\label{eq:1/2}    
\dstart (S_p^n) \le 2 \left(\frac{p^n + p^{[n\ {\rm even}]}}{p+1}\right)  - 1\,.
\end{equation}
Since Sierpi\'{n}ski graphs are traceable (in fact, hamiltonian), this bound confirms Rall's conjecture for the case of Sierpi\'{n}ski graphs asserting that if $G$ is a traceable graph, then $\dstart(G) \le \lceil |V(G)|/2\rceil$, see~\cite[Conjecture 2.41]{book-2021}. Moreover, it also confirms the stronger 1/2-conjecture~\cite[Conjecture 1.2]{bujtas-2022} asserting that as soon as $\delta(G)\ge 2$, then $\dstart(G) \le \lceil |V(G)|/2 \rceil$.

In this section we first significantly improve~\eqref{eq:1/2} by showing that the (Staller start) game domination number of $S_p^n$ is at most $\frac{2(p-1)}{p^2}|V(S_p^n)|$. This implies that, for every fixed $p$, the game domination number of $S_p^n$ is bounded above by a constant fraction of its order, namely $\frac{2(p-1)}{p^2}.$ Next we improve also the lower bound by showing that the (Staller start) game domination number is at least $(2p-3)p^{n-2}$, equivalently $\frac{2p-3}{p^2}$ of its order.

\begin{lemma}
\label{lem:X_p}
    If $p\ge 3$, then the following hold.
    \begin{enumerate}
    \item[(i)] $\dstart(\Gamma_p|X_p) = \sstart(\Gamma_p|X_p) = 2p-2$.
    \item[(ii)] $\dstart(\Gamma_p^-|X_p^-) = \sstart(\Gamma_p^-|X_p^-) = 2p-2$.
    \end{enumerate}
    Moreover, if in each of the cases Staller decides to skip a move, the corresponding game ends in at most $2p-2$ moves.
\end{lemma}

\proof
First, consider the D-game played on $\Gamma_p|X_p$. The strategy of Dominator is to play vertices $1p, 2p\ldots , (p-1)p$ in his first $p-1$ moves. If any of these moves is not possible, then he chooses any other legal move. Note that assuming that these $p-1$ vertices have been played, the only not yet dominated vertex in $\Gamma_p|X_p$ is the vertex $pp$. Hence no matter how Staller played, she is forced to dominate this undominated vertex in her $(p-1)^{\rm{st}}$ move. In this way Dominator achieves the goal to finish the game in not more than $2p-2$ moves. Hence $\dstart(\Gamma_p|X_p) \leq 2p-2$.

We next show that Staller has a strategy that at least $2p-2$ vertices will be played on $\Gamma_p|X_p$. Note first that during the game, if Dominator has played some move in $iS_p^1$, then by the Continuation Principle he will not select any other vertex of $iS_p^1$ in his first $p-1$ moves. By the same reason, he will never play on $X_p$ during these moves. Hence we may assume without loss of generality that no vertex of $pS_p^1$ was selected by Dominator by this stage of the game. Furthermore we may assume without loss of generality that in his $i^{\rm{th}}$ move Dominator plays in $iS_p^1$. Now the strategy of Staller in her first $p-2$ moves is to responds in $iS_p^1$ as soon as Dominator selects the vertex in $iS_p^1$. That is, in her $i^{\rm{th}}$ move she always has a legal move in the set $\{ i(i+1), ip\}$. In her $(p-1)^{\rm{th}}$ move Staller selects the vertex $pp$, because it has not yet been dominated. Thus at least $2p-2$ moves are played and hence $\dstart(\Gamma_p|X_p) \geq 2p-2.$ We can conclude that $\dstart(\Gamma_p|X_p)=2p-2$.

Second, consider the S-game played on $\Gamma_p|X_p$. By the Continuation Principle we may without loss of generality assume that the first vertex selected by Staller is $x_p$. After that the strategy of Dominator is as above, that is, to select vertices $1p,\ldots , (p-1)p$. After that all the vertices are dominated which in turn implies that $\sstart(\Gamma_p|X_p) \leq 2p-2$.
We next show that Staller has a strategy that at least $2p-2$ vertices will be played on $\Gamma_p|X_p$. By the Continuation Principle we may without loss of generality assume that Staller plays $x_p$ in her first move. If Dominator has played some move in $iS_p^1$, then by the Continuation Principle he will not select any other vertex of $iS_p^1$ in his first $p-2$ moves. By the same reason, he will never play on $X_p$ during these moves. The strategy of Staller in her next $p-2$ moves is to respond in $iS_p^1$ as soon as Dominator selects the vertex in $iS_p^1$. Thus after these moves, $1+2(p-2)$ vertices have been selected. Moreover, at this stage of the game not all vertices are dominated (there exists $i\in [p]$ such that no vertex from $iS_p^1$ has been played) and hence Dominator is forced to selects at least one more vertex. This already gives us the required conclusion, that is $\sstart(\Gamma_p|X_p) \geq 2p-2$.

Proceeding in a similar lines as above we obtain that $\dstart(\Gamma_p^-|X_p^-) = \sstart(\Gamma_p^-|X_p^-) = 2p-2$ holds for any $p \geq 3$.

For the case of the game where Staller decides to skip a move, Dominator keeps his above described strategy of playing $1p,2p,\ldots , (p-1)p$. Since Staller skipped at least one move at most $2p-3$ vertices have been played so far. Since at that point only the vertex $pp$ is not yet dominated, we can conclude that at most $2p-2$ vertices are played also if Staller decides to skip a move.
\qed
 
\begin{lemma}
\label{lem:S_p^2|Ex}
    If $p\geq 3$, then $\dstart\left(S_p^2|{\rm Ex}(S_p^2)\right) \geq 2p-3$ and $\sstart\left({S_p^2|\rm Ex}(S_p^2)\right) \geq 2p-3$.
 Moreover, if Dominator decides to skip a move, the corresponding game will not end in fewer than $2p-3$ moves.
\end{lemma}

\proof
Let us start with the D-game played on $S_p^2|{\rm Ex}(S_p^2)$. Consider the first $p-2$ moves of Dominator. By the Continuation Principle without loss of generality we may assume that these $p-2$ vertices belong to $iS_p^1$, where $i\in [p-2]$. The strategy of Staller is then to reply to each move of Dominator in the same copy of $iS_p^1$ which is possible using the same arguments as in the proof of Lemma~\ref{lem:X_p}. Note that each of these $2p-4$ selected vertices dominates at most $2p-4$ vertices from $V((p-1)S_p^1) \cup V(pS_p^1)$. As a consequence at least two vertices are not yet dominated and hence Dominator is forced to play at least one more move. Thus $\dstart\left(S_p^2|{\rm Ex}(S_p^2)\right) \geq 2p-3$.

In the S-game, Staller follows the same strategy as above. By the Continuation Principle we may without loss of generality assume that Staller selects the vertex $11$ in her first move. After that her strategy is as above, that is whenever Dominator selects a vertex from $iS_p^1$, where $i\neq 1$ by the Continuation Principle, she replies by an arbitrary playable vertex from $iS_p^1$. In this way after the first $1+2(p-3)$ moves, there exists at least two copies of $S_p^1$ in $S_p^2|{\rm Ex}(S_p^2)$ such that none of its vertices were selected, say $iS_p^1$ and $jS_p^1$. Hence after the next move of Dominator, say in $iS_p^1$, the vertex $j1$ is still not dominated. Thus Staller will have another move, which implies that $\sstart\left(S_p^2|{\rm Ex}(S_p^2)\right) \geq 2p-3$.

For the case where Dominator decides to skip a move Staller keeps her above described strategy to reply to each move of Dominator in the same copy of $iS_p^1$. Whenever Dominator decides to skip a move Staller plays (if possible) inside those copies of $iS_p^1$ in which a vertex has already been played. If no such legal move exists, then she will play extreme vertex of not yet played $iS_p^1$. By this strategy she can guarantee that at least $2p-3$ vertices are played. 
\qed


\begin{figure}[ht!]
\begin{center}
\begin{tikzpicture}[scale=0.42,style=thick,x=1cm,y=1cm]
\def\vr{5pt}

\begin{scope}[xshift=10cm, yshift=10cm] 

\coordinate(x1) at (0,0);
\coordinate(x2) at (1,1.732);
\coordinate(x3) at (2,3.464);
\coordinate(x4) at (3,5.196);
\coordinate(x5) at (4,3.464);
\coordinate(x6) at (5,1.732);
\coordinate(x7) at (6,0);
\coordinate(x8) at (4,0);
\coordinate(x9) at (2,0);

\coordinate(x10) at (8,0);
\coordinate(x11) at (9,1.732);
\coordinate(x12) at (10,3.464);
\coordinate(x13) at (11,5.196);
\coordinate(x14) at (12,3.464);
\coordinate(x15) at (13,1.732);
\coordinate(x16) at (14,0);
\coordinate(x17) at (12,0);
\coordinate(x18) at (10,0);

\coordinate(x19) at (4,6.928);
\coordinate(x20) at (5,8.66);
\coordinate(x21) at (6,10.392);
\coordinate(x22) at (7,12.124);
\coordinate(x23) at (8,10.392);
\coordinate(x24) at (9,8.66);
\coordinate(x25) at (10,6.928);
\coordinate(x26) at (8,6.928);
\coordinate(x27) at (6,6.928);

\coordinate(x28) at (16,0);
\coordinate(x29) at (17,1.732);
\coordinate(x30) at (18,3.464);
\coordinate(x31) at (19,5.196);
\coordinate(x32) at (20,3.464);
\coordinate(x33) at (21,1.732);
\coordinate(x34) at (22,0);
\coordinate(x35) at (20,0);
\coordinate(x36) at (18,0);

\coordinate(x37) at (24,0);
\coordinate(x38) at (25,1.732);
\coordinate(x39) at (26,3.464);
\coordinate(x40) at (27,5.196);
\coordinate(x41) at (28,3.464);
\coordinate(x42) at (29,1.732);
\coordinate(x43) at (30,0);
\coordinate(x44) at (28,0);
\coordinate(x45) at (26,0);

\coordinate(x46) at (20,6.928);
\coordinate(x47) at (21,8.66);
\coordinate(x48) at (22,10.392);
\coordinate(x49) at (23,12.124);
\coordinate(x50) at (24,10.392);
\coordinate(x51) at (25,8.66);
\coordinate(x52) at (26,6.928);
\coordinate(x53) at (24,6.928);
\coordinate(x54) at (22,6.928);

\coordinate(x55) at (8,13.856);
\coordinate(x56) at (9,15.588);
\coordinate(x57) at (10,17.32);
\coordinate(x58) at (11,19.052);
\coordinate(x59) at (12,17.32);
\coordinate(x60) at (13,15.588);
\coordinate(x61) at (14,13.856);
\coordinate(x62) at (12,13.856);
\coordinate(x63) at (10,13.856);

\coordinate(x64) at (16,13.856);
\coordinate(x65) at (17,15.588);
\coordinate(x66) at (18,17.32);
\coordinate(x67) at (19,19.052);
\coordinate(x68) at (20,17.32);
\coordinate(x69) at (21,15.588);
\coordinate(x70) at (22,13.856);
\coordinate(x71) at (20,13.856);
\coordinate(x72) at (18,13.856);

\coordinate(x73) at (12,20.784);
\coordinate(x74) at (13,22.516);
\coordinate(x75) at (14,24.248);
\coordinate(x76) at (15,25.98);
\coordinate(x77) at (16,24.248);
\coordinate(x78) at (17,22.516);
\coordinate(x79) at (18,20.784);
\coordinate(x80) at (16,20.784);
\coordinate(x81) at (14,20.784);


\draw (x1) -- (x2) -- (x3) -- (x4) -- (x5) -- (x6)--(x7)--(x8)--(x9)--(x1);
\draw (x2) -- (x9);
\draw (x3) -- (x5);
\draw (x6) -- (x8);

\draw (x10) -- (x11) -- (x12) -- (x13) -- (x14) -- (x15)--(x16)--(x17)--(x18)--(x10);
\draw (x11) -- (x18);
\draw (x12) -- (x14);
\draw (x15) -- (x17);

\draw (x19) -- (x20) -- (x21) -- (x22) -- (x23) -- (x24)--(x25)--(x26)--(x27)--(x19);
\draw (x20) -- (x27);
\draw (x21) -- (x23);
\draw (x24) -- (x26);

\draw[dotted, very thick] (x4) -- (x19);
\draw[dotted, very thick] (x25) -- (x13);
\draw[dotted, very thick] (x7) -- (x10);

\draw (x28) -- (x29) -- (x30) -- (x31) -- (x32) -- (x33)--(x34)--(x35)--(x36)--(x28);
\draw (x29) -- (x36);
\draw (x30) -- (x32);
\draw (x33) -- (x35);

\draw (x37) -- (x38) -- (x39) -- (x40) -- (x41) -- (x42)--(x43)--(x44)--(x45)--(x37);
\draw (x38) -- (x45);
\draw (x39) -- (x41);
\draw (x42) -- (x44);

\draw (x46) -- (x47) -- (x48) -- (x49) -- (x50) -- (x51)--(x52)--(x53)--(x54)--(x46);
\draw (x47) -- (x54);
\draw (x48) -- (x50);
\draw (x51) -- (x53);

\draw[dotted, very thick] (x31) -- (x46);
\draw[dotted, very thick] (x52) -- (x40);
\draw[dotted, very thick] (x34) -- (x37);

\draw (x55) -- (x56) -- (x57) -- (x58) -- (x59) -- (x60)--(x61)--(x62)--(x63)--(x55);
\draw (x56) -- (x63);
\draw (x57) -- (x59);
\draw (x60) -- (x62);

\draw (x64) -- (x65) -- (x66) -- (x67) -- (x68) -- (x69)--(x70)--(x71)--(x72)--(x64);
\draw (x65) -- (x72);
\draw (x66) -- (x68);
\draw (x69) -- (x71);

\draw (x73) -- (x74) -- (x75) -- (x76) -- (x77) -- (x78)--(x79)--(x80)--(x81)--(x73);
\draw (x74) -- (x81);
\draw (x75) -- (x77);
\draw (x78) -- (x80);

\draw[dotted, very thick] (x58) -- (x73);
\draw[dotted, very thick] (x79) -- (x67);
\draw[dotted, very thick] (x61) -- (x64);

\draw[dotted, very thick] (x22) -- (x55);
\draw[dotted, very thick] (x70) -- (x49);
\draw[dotted, very thick] (x16) -- (x28);


\draw(x1)[fill=white] circle(\vr) node[below]{{\footnotesize 2222}};
\draw(x2)[fill=white] circle(\vr) node[left]{{\footnotesize 2221}};
\draw(x3)[fill=white] circle(\vr) node[left]{{\footnotesize 2212}};
\draw(x4)[fill=white] circle(\vr) node[left]{{\footnotesize 2211}};
\draw(x5)[fill=white] circle(\vr) node[right]{{\footnotesize 2213}};
\draw(x6)[fill=white] circle(\vr) node[right]{{\footnotesize 2231}};
\draw(x7)[fill=white] circle(\vr) node[below]{{\footnotesize 2233}};
\draw(x8)[fill=white] circle(\vr) node[below]{{\footnotesize 2232}};
\draw(x9)[fill=white] circle(\vr) node[below]{{\footnotesize 2223}};

\draw(x10)[fill=white] circle(\vr) node[below]{{\footnotesize 2322}};
\draw(x11)[fill=white] circle(\vr) node[left]{{\footnotesize 2321}};
\draw(x12)[fill=white] circle(\vr) node[left]{{\footnotesize 2312}};
\draw(x13)[fill=white] circle(\vr) node[right]{{\footnotesize 2311}};
\draw(x14)[fill=white] circle(\vr) node[right]{{\footnotesize 2313}};
\draw(x15)[fill=white] circle(\vr) node[right]{{\footnotesize 2331}};
\draw(x16)[fill=white] circle(\vr) node[below]{{\footnotesize 2333}};
\draw(x17)[fill=white] circle(\vr) node[below]{{\footnotesize 2332}};
\draw(x18)[fill=white] circle(\vr) node[below]{{\footnotesize 2323}};

\draw(x19)[fill=white] circle(\vr) node[left]{{\footnotesize 2122}};
\draw(x20)[fill=white] circle(\vr) node[left]{{\footnotesize 2121}};
\draw(x21)[fill=white] circle(\vr) node[left]{{\footnotesize 2112}};
\draw(x22)[fill=white] circle(\vr) node[left]{{\footnotesize 2111}};
\draw(x23)[fill=white] circle(\vr) node[right]{{\footnotesize 2113}};
\draw(x24)[fill=white] circle(\vr) node[right]{{\footnotesize 2131}};
\draw(x25)[fill=white] circle(\vr) node[right]{{\footnotesize 2133}};
\draw(x26)[fill=white] circle(\vr) node[below]{{\footnotesize 2132}};
\draw(x27)[fill=white] circle(\vr) node[below]{{\footnotesize 2123}};

\draw(x28)[fill=white] circle(\vr) node[below]{{\footnotesize 3222}};
\draw(x29)[fill=white] circle(\vr) node[left]{{\footnotesize 3221}};
\draw(x30)[fill=white] circle(\vr) node[left]{{\footnotesize 3212}};
\draw(x31)[fill=white] circle(\vr) node[left]{{\footnotesize 3211}};
\draw(x32)[fill=white] circle(\vr) node[right]{{\footnotesize 3213}};
\draw(x33)[fill=white] circle(\vr) node[right]{{\footnotesize 3231}};
\draw(x34)[fill=white] circle(\vr) node[below]{{\footnotesize 3233}};
\draw(x35)[fill=white] circle(\vr) node[below]{{\footnotesize 3232}};
\draw(x36)[fill=white] circle(\vr) node[below]{{\footnotesize 3223}};

\draw(x37)[fill=white] circle(\vr) node[below]{{\footnotesize 3322}};
\draw(x38)[fill=white] circle(\vr) node[left]{{\footnotesize 3321}};
\draw(x39)[fill=white] circle(\vr) node[left]{{\footnotesize 3312}};
\draw(x40)[fill=white] circle(\vr) node[right]{{\footnotesize 3311}};
\draw(x41)[fill=white] circle(\vr) node[right]{{\footnotesize 3313}};
\draw(x42)[fill=white] circle(\vr) node[right]{{\footnotesize 3331}};
\draw(x43)[fill=white] circle(\vr) node[below]{{\footnotesize 3333}};
\draw(x44)[fill=white] circle(\vr) node[below]{{\footnotesize 3332}};
\draw(x45)[fill=white] circle(\vr) node[below]{{\footnotesize 3323}};

\draw(x46)[fill=white] circle(\vr) node[left]{{\footnotesize 3122}};
\draw(x47)[fill=white] circle(\vr) node[left]{{\footnotesize 3121}};
\draw(x48)[fill=white] circle(\vr) node[left]{{\footnotesize 3112}};
\draw(x49)[fill=white] circle(\vr) node[right]{{\footnotesize 3111}};
\draw(x50)[fill=white] circle(\vr) node[right]{{\footnotesize 3113}};
\draw(x51)[fill=white] circle(\vr) node[right]{{\footnotesize 3131}};
\draw(x52)[fill=white] circle(\vr) node[right]{{\footnotesize 3133}};
\draw(x53)[fill=white] circle(\vr) node[below]{{\footnotesize 3132}};
\draw(x54)[fill=white] circle(\vr) node[below]{{\footnotesize 3123}};

\draw(x55)[fill=white] circle(\vr) node[left]{{\footnotesize 1222}};
\draw(x56)[fill=white] circle(\vr) node[left]{{\footnotesize 1221}};
\draw(x57)[fill=white] circle(\vr) node[left]{{\footnotesize 1212}};
\draw(x58)[fill=white] circle(\vr) node[left]{{\footnotesize 1211}};
\draw(x59)[fill=white] circle(\vr) node[right]{{\footnotesize 1213}};
\draw(x60)[fill=white] circle(\vr) node[right]{{\footnotesize 1231}};
\draw(x61)[fill=white] circle(\vr) node[below]{{\footnotesize 1233}};
\draw(x62)[fill=white] circle(\vr) node[below]{{\footnotesize 1232}};
\draw(x63)[fill=white] circle(\vr) node[below]{{\footnotesize 1223}};

\draw(x64)[fill=white] circle(\vr) node[below]{{\footnotesize 1322}};
\draw(x65)[fill=white] circle(\vr) node[left]{{\footnotesize 1321}};
\draw(x66)[fill=white] circle(\vr) node[left]{{\footnotesize 1312}};
\draw(x67)[fill=white] circle(\vr) node[right]{{\footnotesize 1311}};
\draw(x68)[fill=white] circle(\vr) node[right]{{\footnotesize 1313}};
\draw(x69)[fill=white] circle(\vr) node[right]{{\footnotesize 1331}};
\draw(x70)[fill=white] circle(\vr) node[right]{{\footnotesize 1333}};
\draw(x71)[fill=white] circle(\vr) node[below]{{\footnotesize 1332}};
\draw(x72)[fill=white] circle(\vr) node[below]{{\footnotesize 1323}};

\draw(x73)[fill=white] circle(\vr) node[left]{{\footnotesize 1122}};
\draw(x74)[fill=white] circle(\vr) node[left]{{\footnotesize 1121}};
\draw(x75)[fill=white] circle(\vr) node[left]{{\footnotesize 1112}};
\draw(x76)[fill=white] circle(\vr) node[above]{{\footnotesize 1111}};
\draw(x77)[fill=white] circle(\vr) node[right]{{\footnotesize 1113}};
\draw(x78)[fill=white] circle(\vr) node[right]{{\footnotesize 1131}};
\draw(x79)[fill=white] circle(\vr) node[right]{{\footnotesize 1133}};
\draw(x80)[fill=white] circle(\vr) node[below]{{\footnotesize 1132}};
\draw(x81)[fill=white] circle(\vr) node[below]{{\footnotesize 1123}};

\end{scope}
\end{tikzpicture}
\caption{The Sierpi\'nski graph $S_3^4$. The dotted edges are the connecting edges between the copies of $S_3^2$ in it.}
\label{fig:S_3^4}
\end{center}
\end{figure}
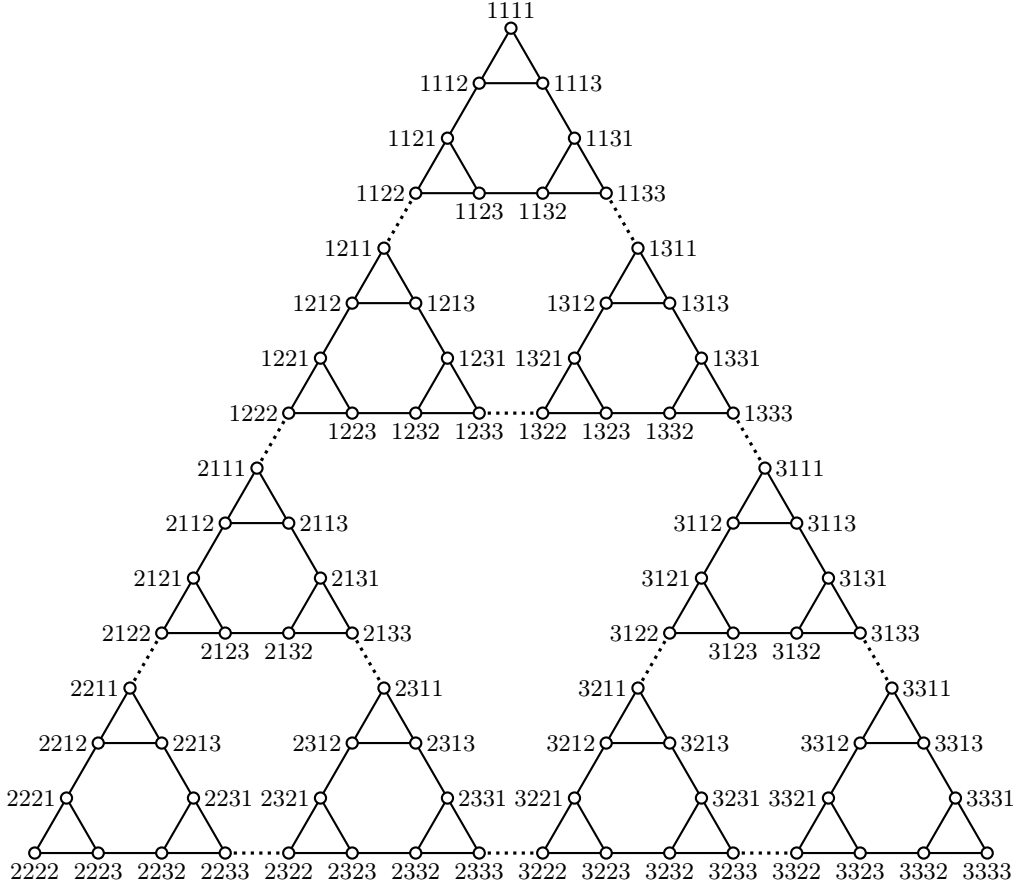

\begin{theorem}
    \label{thm:mainDom_game}
    If $p\geq 3$ and $n\ge 2$, then 
    $(2p-3)p^{n-2} \leq \dstart(S_p^n), \sstart(S_p^n) \leq (2p-2)p^{n-2}.$    
\end{theorem}

\proof
The statement can be easily checked for $n=2$, hence we assume in what follows that $n \geq 3$. First we prove that $\dstart(S_p^n) \leq (2p-2)p^{n-2}$. Fix $n\ge 3$ and $p\ge 3$, and set
$$E_2=\{uv \in E(S_p^n);\ u_1=v_1, \dots, u_{h-1} = v_{h-1}, u_h \ne v_h,  h\leq n-2\}.$$ 
Note that $E_2$ contains the edges that connect different copies of $S_p^2$, see Figure~\ref{fig:S_3^4}, where these edges are dotted in the case $S_3^4$. Let $G_{n,p}$ be the graph obtained from $G=S_p^n$ by inductively using operation $G_{uv}$ for every edge $uv \in E_2$. Note that $G_{n,p}$  consists of $p^{n-2}$ components, $p$ of them being isomorphic to $\Gamma_p^-$ and the rest isomorphic to $\Gamma_p$.
By Theorem~\ref{thm:cutting-lemma} it follows that $\dstart(S_p^n) \leq \dstart(G_{n,p})$. Thus, it remains to show that $\dstart(G_{n,p}) \leq (2p-2)p^{n-2}$. For this sake consider the strategy of Dominator in which he always optimally plays, if possible, in the same component of $G_{n,p}$ in which Staller played the last move. Note that such a move is not possible only if all the vertices of the component in which Staller played last are already dominated. In such a case, provided the game is not over yet, Dominator either optimally plays in a component in which no vertex was played or in a component in which he had a last move. Note that by the strategy of Dominator, Staller will never play two consecutive moves in the same component of $G_{n,p}$. By applying Lemma~\ref{lem:X_p} it follows that at most $2p-2$ moves will be played in each component of $G_{n,p}$. As there are $p^{n-2}$ such components, we can conclude that $\dstart(S_p^n) \leq \dstart(G_{n,p}) \leq (2p-2)p^{n-2}$. An analogous strategy of Dominator yields the required result for the S-game.

Now we prove $\dstart(S_p^n) \geq (2p-3)p^{n-2}$. We partition $V(S_p^n)$ into $p^{n-2}$ sets where each set induces a copy of $S_p^2$. Staller's strategy is the following. She always optimally plays, if possible, in the same partition set in which Dominator played his last move, where optimal play is with respect to the graph $S_p^2|{\rm{Ex}(S_p^2)}$. Note that such a move is not possible only if all the vertices of the partition set in which Dominator played last are already dominated. In such a case, provided the game is not over yet, Staller either optimally plays in a partition set in which no vertex was played or in a partition set in which she had a last move. Note that by the strategy of Staller, Dominator will never play two consecutive moves in the same partition set of $S_p^n$. By applying Lemma~\ref{lem:S_p^2|Ex} and by the Continuation Principle it follows that at least $2p-3$ moves will be played in each of the partition sets. Thus $\dstart(S_p^n) \geq (2p-3)p^{n-2}$. An analogous strategy of Staller yields $\sstart(S_p^n) \geq (2p-3)p^{n-2}$.
\qed

\section{Total domination game}
\label{sec:total-dom-game}

In this section, we study the total domination game on Sierpi\'{n}ski graphs. We will make use of the Total Continuation Principle for the total domination game, whose formulation is identical to that for the Continuation Principle, with $\dstart(G)$ and $\sstart(G)$ respectively replaced by $\tdstart(G)$ and $\tsstart(G)$.

As with the domination number, determining the total domination number is 
computationally hard for general graphs~\cite{GJ-79}. For Sierpi\'{n}ski 
graphs, however, this parameter can be determined efficiently. In~\cite{GKM-2013} 
it was shown that
\begin{equation}
\label{eq:total}  
\gammat(S_p^n)=
\begin{cases}
p^{n-1}; & \text{if } p \text{ is even},\\
p^{n-1}+1; & \text{if } p \text{ is odd}.
\end{cases}
\end{equation}
On the other hand, for the Grundy total domination number, which corresponds to the maximum length of a legal total dominating sequence, it was shown in~\cite[Corollary 3.59]{kos-2019} that 
\begin{equation}
\label{eq:grtotal}  
\grundyt(S_p^n) \geq 
p^{n-1}+\frac{p(p^{n-1}-1)}{2}.
\end{equation}
Next we consider the game total domination number on Sierpi\'{n}ski graphs and show that it strictly lies between the above two extreme options, in which either just Dominator (in~\eqref{eq:total}) or just Staller (in~\eqref{eq:grtotal}) is playing the game.

The relation $\gamma_{{\rm{tg}}}(G) \leq 2\gammat(G) -1$ from~\cite[Theorem 3.1]{hkr-2015} applied to Sierpi\'{n}ski graphs yields:
\[
\tdstart(S_p^n)\leq 
\begin{cases}
2p^{n-1}-1; & \text{if } p \text{ is even},\\
2p^{n-1}+1; & \text{if } p \text{ is odd}.
\end{cases}
\]

Next, we improve a trivial lower bound $\tdstart(S_p^n) \geq \gammat(S_p^n)$ as follows. We start with a partial result that will be needed in the proof.

\begin{lemma}
\label{lem:S_p^2|Ex_total}
    If $p\geq 3$, then $\tdstart\left(S_p^2|{\rm Ex}(S_p^2)\right) \geq 2p-2$ and $\tsstart\left({S_p^2|\rm Ex}(S_p^2)\right) \geq 2p-2$.
 Moreover, if Dominator decides to skip a move, the corresponding game will last at least $2p-2$ moves.
\end{lemma}

\proof
Let us start with the D-game played on $S_p^2|{\rm Ex}(S_p^2)$. Staller plays using the following strategy. If Dominator plays in $iS_p^1$ for some $i\in [p]$, then Staller selects $ii$, if possible, otherwise she selects an arbitrary vertex from $V(iS_p^1)\setminus \{ii\}$. Note that by this strategy she totally dominates exactly one new vertex in each of her moves. Consider now the first $p-1$ moves of Dominator. By the Total Continuation Principle and Staller's strategy without loss of generality we may assume that these $p-1$ vertices belong to $iS_p^1$, where $i\in [p-1]$, and each move of Dominator can totally dominate at most $p-1$ vertices previously not totally dominated. After $2p-3$ moves, and including the $p$ pre-totally dominated vertices, at most $(p-1)^2 + (p-2) + p = p^2 - 1$ vertices have been totally dominated. At least one vertex is not yet totally dominated and hence Staller is able to play at least one more move. Thus $\tdstart\left(S_p^2|{\rm Ex}(S_p^2)\right) \geq 2p-2$.

In the S-game, let Staller, by her strategy, first play vertex $11$ and thus she totally dominates $p-1$ vertices that were not previously totally dominated. Her strategy for the next $p-3$ moves is as above, that is whenever Dominator selects a vertex from $iS_p^1$, where $i\neq 1$ by the Total Continuation Principle, she replies by an arbitrary playable vertex from $iS_p^1$. In this way,
after $2p-5$ moves there are at least two indices $i,j \in [p]\setminus \{1\}$, such that no vertex from $iS_p^1$ and $jS_p^1$ was played until this stage of the game. By the Total Continuation Principle, Dominator in his next move chooses a vertex from $V(iS_p^1) \cup V(jS_p^1)$, say a vertex from $iS_p^1$. Hence after $2p-4$ moves, at most $(p-1)+(p-3)+(p-1)(p-2) +p = p^2 -2$ vertices are totally dominated.  In particular, the vertices $j1$ and either $ji$ or $ij$ are not yet totally dominated. With Staller's next move, she plays vertex $1j$.  Since this leaves either $ji$ or $ij$ not yet totally dominated, the game continues.  This implies that $\tsstart\left(S_p^2|{\rm Ex}(S_p^2)\right) \geq 2p-2$.

For the case of the game where Dominator decides to skip a move, Staller keeps her above described strategy to reply to each move of Dominator in the same copy of $iS_p^1$. Whenever Dominator decides to skip a move Staller plays (if possible) inside those copies of $iS_p^1$ in which a vertex has already been played. If no such legal move exists, then she will play the extreme vertex of a not yet played $iS_p^1$. By this strategy she can guarantee that at least $2p-2$ vertices are played. 
\qed

\begin{theorem}\label{thm:mainTotal}
    If $p \geq 3$ and $n\geq 2$, then $\tdstart(S_p^n),\tsstart(S_p^n) \geq (2p-2)p^{n-2}$.
\end{theorem}
\proof
Consider first the case $n=2$. Then the Total Continuation Principle implies that $\tdstart(S_p^2) \geq \tdstart\left(S_p^2|{\rm{Ex}}(S_p^2)\right)$, as well as $\tsstart(S_p^2) \geq \tsstart(S_p^2|{\rm{Ex}}(S_p^2))$. Furthermore, $\tdstart\left(S_p^2|{\rm{Ex}}(S_p^2)\right) \geq 2p-2$ and $\tsstart(S_p^2|{\rm{Ex}}(S_p^2)) \geq 2p-2$  by  Lemma~\ref{lem:S_p^2|Ex_total}, which finishes the case $n=2$. Thus we may now assume that $n \geq 3$.

We partition $V(S_p^n)$ into $p^{n-2}$ sets where each set induces a copy of $S_p^2$. Staller's strategy is the following. She always optimally plays, if possible, in the same partition set in which Dominator played his last move, where optimal play is with respect to the graph $S_p^2|{\rm{Ex}(S_p^2)}$. Note that such a move is not possible only if all the vertices of the partition set in which Dominator played last are already dominated. In such a case, provided the game is not over yet, Staller either optimally plays in a partition set in which no vertex was played or in a partition set in which she had the last move. Note that by the strategy of Staller, Dominator will never play two consecutive moves in the same partition set of $S_p^n$. By applying Lemma~\ref{lem:S_p^2|Ex_total} and the Total Continuation Principle it follows that at least $2p-2$ moves will be played in each of the partition sets. Thus $\tdstart(S_p^n) \geq (2p-2)p^{n-2}$. An analogous strategy of Staller yields $\tsstart(S_p^n) \geq (2p-2)p^{n-2}$.
\qed

The three parameters therefore exhibit three different scales of behaviour. The total domination number is equal to $p^{n-1}$, while our results show that
\[
\left(2-\frac{2}{p}\right)p^{n-1}
\leq
\tdstart(S_p^n)
\leq
2p^{n-1}-1.
\]
Thus, the game total domination number is confined to a narrow interval around $2p^{n-1}$. In contrast, the known lower bound for the Grundy total domination number yields
\[
\grundyt(S_p^n)
\geq
\left(1+\frac{p}{2}\right)p^{n-1}-\frac{p}{2}.
\]
Consequently, the parameters $\gammat$, $\tdstart$, and $\grundyt$ can be roughly viewed as being of sizes
$$p^{n-1}, 2p^{n-1},\left(1+\frac{p}{2}\right)p^{n-1},$$
respectively. In particular, whereas the total domination game increases the size of a minimum total dominating sequence by at most a factor of $2$, the Grundy total domination number may be larger by a factor that grows linearly with $p$.

\section*{Acknowledgements}

Tanja Dravec and Sandi Klav\v zar were supported by the Slovenian Research and Innovation Agency (ARIS) under the grants P1-0297, N1-0285, N1-0355, N1-0431, J1-70045.  Daniel P. Johnston would like to thank the financial support provided by an AMS-Simons Research Enhancement Grant for Primarily Undergraduate Institution Faculty.

\begin{thebibliography}{99}

\bibitem{att-2026}
F.~Attarzadeh, A.~Abbasi, A.~Behtoei, 
Girth and planarity of the generalized {Sierpi{\'n}ski} gasket {{\(S[G, t]\)}},
J. Algebra Relat. Top. 14 (2026) 125--137.

\bibitem{bgk-2016} 
B.\ Bre\v sar, T.\ Gologranc, T.\ Kos, 
Dominating sequences under atomic changes with applications in Sierpi\'{n}ski and interval graphs, 
Appl.\ Anal.\ Discrete Math.\ 10 (2016) 518--531.

\bibitem{bgmrr-2014}
B.~Bre{\v{s}}ar, T.~Gologranc, M.~Milani\v c, D.~F.~Rall, R.~Rizzi,
Dominating sequences in graphs,
Discrete Math. 336 (2014) 22--36.

\bibitem{book-2021} 
B.\ Bre\v{s}ar, M.A.~Henning, S.\ Klav\v zar, D.F.\ Rall, 
Domination Games Played on Graphs,
Springer, Cham (2021).

\bibitem{bresar-2016} 
B.\ Bre\v sar, M.A.~Henning, D.F.~Rall, 
Total dominating sequences in graphs,
Discrete Math. 339 (2016) 1665--1676.

\bibitem{bresar-2010} 
B.\ Bre\v{s}ar, S.\ Klav\v zar, D.F.\ Rall, 
Domination game and an imagination strategy, 
SIAM J.\ Discrete Math.\ 24 (2010) 979--991.

\bibitem{brito-2026}
J.M.~Brito, T.~Marcilon, N.A.~Martins, R.~Sampaio, 
The normal domination game in graphs,
J. Comput. Syst. Sci. 157 (2026) Paper 103751.

\bibitem{bujtas-2022}
Cs.~Bujt\'{a}s, V.~Ir\v{s}i\v{c}, S.~Klav\v{z}ar, 
$1/2$-conjectures on the domination game and claw-free graphs,
European J.\ Combin.\ 101 (2022) Paper 103467.

\bibitem{dorbec_cutting_2019} 
P.\ Dorbec, M.A.\ Henning, S.\ Klav\v zar, G.\ Ko\v smrlj, Cutting lemma and union lemma for the domination game,
Discrete Math.\ 342 (2019) 1213--1222.

\bibitem{forcan-2022}
J.~Forcan, M.~Mikala\v{c}ki, 
Maker-Breaker total domination game on cubic graphs,
Discrete Math.\ Theor.\ Comput.\ Sci.\ 24 (2022) Paper 20.

\bibitem{fuhrer-2026+}
J.~F\"{u}hrer, G.~Grasegger, P.~Hametner, O.~Roche-Newton,
The minimum degree question for the {M}aker-{B}reaker domination game,
\url{arXiv:2606.04824}  [math.CO] (2026).

\bibitem{GJ-79} 
M.R.\ Garey, D.S.\ Johnson, 
Computers and Intractability: A Guide to the Theory of NP-Completeness, 
Freeman, New York, 1979.

\bibitem{GKM-2013} 
S.~Gravier, M.~Kov\v se, M.~Mollard, J.~Moncel, A.~Parreau,
New results on variants of covering codes in Sierpi\'{n}ski graphs, 
Des.\ Codes Cryptogr.\ 69 (2013) 181--188.

\bibitem{hkr-2015}
M.A.~Henning, S.~Klav{\v{z}}ar, D.F.~Rall, 
Total version of the domination game,
Graphs Combin.\ 31 (2015) 1453--1462.

\bibitem{henning-2017}
M.A.~Henning, S.~Klav\v{z}ar, D.F.~Rall,
The $4/5$ upper bound on the game total domination number,
Combinatorica 37 (2017) 223--251.

\bibitem{hs-2022}
A.M.~Hinz, P.K.~Stockmeyer, 
Precious metal sequences and {Sierpi{\'n}ski}-type graphs,
J. Integer Seq. 25 (2022) Paper 22.4.8.

\bibitem{irsic-2025}
V.~Ir\v{s}i\v{c} Chenoweth, 
Complexity of the game connected domination problem, 
Theor. Comput. Sci. 1057 (2025) Paper 115559.

\bibitem{james-2023}
T.~James, A.~Vijayakumar, 
Domination game: effect of edge contraction and edge subdivision,
Discuss.\ Math.\ Graph Theory 43 (2023) 313--329.

\bibitem{jos-2025}
P.L.~Joshwa, S.~Rajan, T.M.~Rajalaxmi, I.N.~Cangul, 
Embedding of extended {Sierpinski} networks {{\(S^{++}(k, m)\)}} into certain trees,
RAIRO Oper. Res. (2025) 2279--2301.

\bibitem{kinnersley-2013} 
W.B.~Kinnersley, D.B.~West, R.~Zamani, 
Extremal problems for game domination number, 
SIAM J.\ Discrete Math. 27 (2013) 2090--2107.

\bibitem{kmz-2017}
S.~Klav\v{z}ar, A.M.~Hinz, S.S.~Zemlji\v{c}, 
A survey and classification of Sierpi\'{n}ski-type graphs, Discrete Appl.\ Math. 217 (2017) 565--600.

\bibitem{km-1997} 
S.~Klav\v{z}ar, U.~Milutinovi\'c, 
Graphs $S(n, k)$ and a variant of the Tower of Hanoi problem,
Czechoslovak Math.\ J.\ 47 (1997) 95--104.

\bibitem{kmp-2002} 
S.~Klav\v{z}ar, U.~Milutinovi\'c, C.~Petr, 
$1$-perfect codes in Sierpi\'nski graphs,
Bull.\ Aust.\ Math.\ Soc.\ 66 (2002) 369--384.

\bibitem{kos-2019}
T.\ Kos, 
Contributions to the Study of Contemporary Domination Invariants of Graphs, 
PhD Thesis, University of Maribor, 2019. \\ \url{https://dk.um.si/Dokument.php?id=137895&lang=slv}

\bibitem{liu-2021}
C.-A.~Liu, 
Roman domination and double {Roman} domination numbers of {Sierpi{\'n}ski} graphs {{\(S(K_n,t)\)}},
Bull. Malays. Math. Sci. Soc. 44 (2021) 4043--4058.

\bibitem{mcs-2023}
M.K.~Menon, M.R.~Chithra, K.S.~Savitha, 
Security in {Sierpi{\'n}ski} graphs,
Discrete Appl.\ Math.\ 328 (2023) 10--15.

\bibitem{portier-2025}
J.~Portier, L.V.~Versteegen, 
A proof of the 3/4-conjecture for the total domination game,
SIAM J.\ Discrete Math.\ 39 (2025) 1--18.

\bibitem{versteegen-2024}
L.~Versteegen, 
A proof of the 3/5-conjecture in the domination game,
European J. Combin. 122 (2024) Paper 104034.

\bibitem{wor-2024}
C.~Worawannotai, K.~Charoensitthichai, 
4-total domination game critical graphs,
Discrete Math. Algorithms Appl. 16 (2024) Paper 2350061.

\bibitem{yang-2025}
C.~Yang, P.~Li, Y.~Mao, E.~Cheng, R.~Klasing, 
Constructing disjoint {Steiner} trees in {Sierpi{\'n}ski} graphs,
Fundam. Inform. 194(1) (2025) Paper 2.

\end{thebibliography}
\end{document}